\documentclass[11pt]{article}
\usepackage{euscript}
\usepackage{amsmath}
\usepackage{amssymb}
\usepackage{graphicx}
\usepackage{color}
\definecolor{dgreen}{rgb}{0,.6,0}

\usepackage{theorem}
\usepackage{amsmath}
\usepackage{amsfonts}
\usepackage{amssymb}
\usepackage{latexsym}
\usepackage{rotating}
\newtheorem{theorem}{\bf Theorem}[section]
\newtheorem{lemma}[theorem]{\bf Lemma}

\newtheorem{coro}[theorem]{\bf Corollary}

\newtheorem{defn}[theorem]{\bf Definition}

\newenvironment{proof}{\noindent{\em Proof:}}{\quad \hfill$\Box$\vspace{2ex}}

\def\TT{{\mathbb T}}

\def\ZZ{{\mathbb Z}}
\def\NN{{\mathbb N}}
\def\RR{{\mathbb R}}
\def\EEm{{\mathbb E}^m}

\def\NNd{{\mathbb N}^d}
\def\RRd{{\mathbb R}^d}

\def\ZZd{{\mathbb Z}^d}
\def\ZZdp{{\mathbb Z}^d_+}
\def\TTd{{\mathbb T}^d}

\def\k2{K^\varphi_2(\TT)}
\def\kpd{\mathcal{H}_{\lambda, p}(\bT^d)}

\def \bN {\Bbb N}
\def \bZ {\Bbb Z}
\def \bE {\Bbb E}

\def \bR {\Bbb R}

\def \bH {\Bbb H}

\def \bB {\Bbb B}
\def \bC {\Bbb C}

\def \bT {\Bbb T}

\def \bY {\Bbb Y}

\def \and {\, \mbox{\rm and}\, }

\def \Re {\,{\rm Re}\,}
\def \Im {\,{\rm Im}\,}

\makeatletter

\newcommand{\Rmnum}[1]{\expandafter\@slowromancap\romannumeral #1@}
\makeatother

\newcommand{\bc}{{\boldsymbol{c}}}

\newcommand{\bj}{{\boldsymbol{j}}}
\newcommand{\bk}{{\boldsymbol{k}}}

\newcommand{\bm}{{\boldsymbol{m}}}

\newcommand{\bp}{{\boldsymbol{p}}}
\newcommand{\bq}{{\boldsymbol{q}}}

\newcommand{\bs}{{\boldsymbol{s}}}
\newcommand{\bt}{{\boldsymbol{t}}}
\newcommand{\bu}{{\boldsymbol{u}}}
\newcommand{\bx}{{\boldsymbol{x}}}
\newcommand{\by}{{\boldsymbol{y}}}
\newcommand{\bw}{{\boldsymbol{w}}}
\newcommand{\bz}{{\boldsymbol{z}}}

\def\TT{{\mathbb T}}

\def\ZZ{{\mathbb Z}}
\def\NN{{\mathbb N}}
\def\RR{{\mathbb R}}
\def\EEm{{\mathbb E}^m}

\def\NNd{{\mathbb N}^d}
\def\RRd{{\mathbb R}^d}

\def\ZZd{{\mathbb Z}^d}
\def\ZZdp{{\mathbb Z}^d_+}
\def\TTd{{\mathbb T}^d}

\newlength{\fixboxwidth}
\begin{document}
\title{\sffamily Approximation by linear combinations of translates of a single function}
\author{Dinh D\~ung$^a $
   and Vu Nhat Huy$^{b,c}$ \\\\
$^a$ Vietnam National University, Information Technology Institute \\
144 Xuan Thuy, Hanoi, Vietnam  \\\\
$^b$ Hanoi University of Science,  Vietnam National University\\
334 Nguyen Trai, Thanh Xuan, Hanoi, Vietnam\\\\
$^c$ TIMAS, Thang Long University\\
Nghiem Xuan Yem, Hoang Mai, Hanoi, Vietnam\\\\
}
 \tolerance 2500
\maketitle

\begin{abstract}
{We study approximation of periodic functions by arbitrary linear combinations of $n$ translates of a single function.} We construct some linear methods of
this approximation for {univariate} functions in the class induced by the convolution with a single function,
and prove upper bounds of  the $L^p$-approximation convergence rate by these methods, when $n \to \infty$, for $1 \leq p \leq \infty$. We also generalize these results to classes of multivariate functions defined {as} the convolution with the tensor product of a single function.  In the case $p=2$, for this class, we also prove a lower bound of the quantity characterizing  best approximation of by arbitrary linear combinations of $n$ translates of arbitrary function. 


\end{abstract}


\section{Introduction}

{The present paper continues investigating the problem of function approximation  by arbitrary linear combinations of $n$ translates of a single function  which has been studied in  \cite{1,DH1}.} In the last papers, some linear methods were constructed for
approximation of {periodic} functions in a class induced by the convolution with a given function, and prove upper bounds of  the $L^p$-approximation convergence rate by these methods,  when $n \to \infty$, for the case $1 < p < \infty$. The main technique of { the proofs} of the results is based on Fourier analysis, in particular, the multiplier theory. However, this technique cannot be extended to the two {important} cases $p = 1$ and $p = \infty$. In the present paper, {we  aim at} this approximation problem for the cases $p = 1$ and $p = \infty$ by using a different technique. For convenience of presentation we will do this for $1 \le p \le \infty$.

 We shall begin our discussion here by introducing  notation used throughout the paper. In this  regard,  we merely 
follow  closely  the presentation in \cite{1,DH1}.
The $d$-dimensional torus denoted by $\bT^d$ is the cross product of
$d$ copies of the interval $[0,2\pi]$ with the identification of the end points. When $d=1$, we merely
denote the $d$-torus by $\bT$. Functions on $\bT^d$
are identified with functions on $\RRd$ which are $2\pi$ periodic in
each variable. Denote by $L^p(\TTd), \ 1 \le p \leq \infty$,
the space of integrable functions on $\TTd$ equipped with the norm
\begin{equation}\nonumber
\|f\|_p \ := 
\begin{cases}
(2\pi)^{-d/p}\left(\int_{\TTd} |f(\bx)|^p d \bx\right)^{1/p}, \ &1\leq p<\infty,\\
\text{ess sup}_{\bx \in \TTd} |f(\bx)|, \ & p=\infty.
\end{cases}
\end{equation}
We will consider only real valued functions on $\TTd$. However, all the results in this paper are true for the complex setting. Also, we will use  Fourier series of a real valued function in complex form.  

Here, we use the notation $\mathbb{N}_m$ for
the set $\{1,2,\ldots,m\}$. For vectors $\bx:=(x_l:l\in \mathbb{N}_d)$ and
$\by:=(y_l:l\in \mathbb{N}_d)$ in $\bT^d$ we use
$(\bx,\by):=\sum_{l\in \mathbb{N}_d}x_ly_l$ for the inner product
of $\bx$ with $\by$. 
 Also, for notational
convenience we allow $\mathbb{N}_0$ and $\mathbb{Z}_0$ to stand for the empty set. Given any integrable function
$f$ on $\TTd$ and any lattice vector $\bj=(j_l: l\in \mathbb{N}_d) \in
\ZZd$, we let ${\widehat f}(\bj)$ denote the $\bj$-th Fourier
coefficient of $f$ defined by the equation 
\[
{\widehat f}(\bj) \ := \ (2\pi)^{-d}\int_{\TTd} f(\bx) \,
e^{-i(\bj,\bx) }\, d\bx.
\]
 Frequently, we use the superscript notation $\bB^d$ to  denote
the cross product of $d$ copies of  a given set $\bB$ in $\mathbb{R}^d$. 

Let $S^{'} (\TTd)$ be the space of {distributions} on $\TTd$. Every $f\in S^{'} (\TTd)$ can be identified with the formal Fourier series 
$$
f= \sum_{\bj \in \ZZd} \widehat{f}(\bj) e^{i (\bj, .)},
$$
where the sequence $(\widehat{f}(\bj): \quad \bj \in \ZZd)$ forms a {tempered sequence}.

Let $\lambda: \mathbb{R} \to \mathbb{R} \setminus \{0\}$   be a bounded function. With the univariate  $\lambda $  we associate the multivariate tensor product function  
${\lambda_d}$  given by 
\[ {\lambda_d}({\bx})
 := \
\prod_{l=1}^d \, \lambda (x_l), \quad {\bx}= (x_l:l\in \mathbb{N}_d),
\]
 and introduce the function $\varphi_{\lambda,d}$,   defined on $\bT^d$ by
the equation
\begin{equation}\label{varphi}
\varphi_{\lambda, d}(\bx) \ := \ \sum_{\bj \in \ZZd}
{\lambda_d}({\bj}) \, e^{i({\bj},{\bx})}.
\end{equation}
 Moreover,  in the case that $d=1$ we merely  write $\varphi_\lambda$ for the univariate function $\varphi_{\lambda,1}$.  
We introduce  a subspace of $L^p(\TTd)$  defined as 
\[\mathcal{H}_{\lambda, p}(\bT^d):=\left\{f: f =
\varphi_{\lambda,d}*g, \ g \in  L^p(\TTd)\right\},
\]
 with norm
\[
\|f\|_{\mathcal{H}_{\lambda, p}(\TT^d)} \ := \|{g}\|_p,
\]
where $f_1*f_2$ is 
the convolution of two functions $f_1$ and $f_2$ on
$\bT^d$.

 {As in  \cite{1,DH1}}, we are concerned with the following concept. 
Let {$\mathbb{W}$ be a prescribed subset of $L^p(\TTd)$ and $\psi \in L^p(\TTd)$ be a
given function.} We are interested in the approximation in
$L^p(\TTd)$-norm of all functions $f \in \mathbb{W}$ by arbitrary
linear combinations of $n$ translates of the function $\psi$,
that is,  by the  functions  in the set $\left\{ \psi(\cdot - {\by}_l): \ {\by}_l
\in \TTd, l\in \mathbb{N}_n \right\}$ and measure the error in terms of the quantity
\begin{equation}  \nonumber
M_n(\mathbb{W},\psi)_p
 := \
\sup_{f \in \mathbb{W}} \ \inf\Bigg\{ \ \Bigg\|f - \sum_{l\in \mathbb{N}_n} c_l \psi(\cdot -
{\by}_l)\Bigg\|_p:c_l\in\mathbb{R},{\by}_l\in\TTd\Bigg\}.
\end{equation}
The aim of the present paper is to investigate the convergence rate,
when $n\rightarrow\infty$, of  $M_n(U_{\lambda, p}(\bT^d), \psi)_p$ {for $1 \le p \le \infty$,}
where 
\[
U_{\lambda, p}(\bT^d):=\left\{f\in \mathcal{H}_{\lambda, p}(\bT^d): \quad \|f\|_{\mathcal{H}_{\lambda, p}(\bT^d)}\leq 1  \right\}
\]
 is the unit ball in $\mathcal{H}_{\lambda, p}(\bT^d)$.  We shall
 also obtain a lower bound for the convergence rate as
 $n\rightarrow\infty$ of the quantity
 \begin{equation}  \nonumber
 M_n(U_{\lambda, 2}(\bT^d))_2  := \
 \inf\left\{ M_n(U_{\lambda, 2}(\bT^d),\psi)_2:\psi \in L^2(\TTd)\right\},
 \end{equation}
 which
 gives information about the best choice of $\psi$.

This paper is organized in the following manner.   
In Section \ref{Univariate approximations}, we give the necessary background
from Fourier analysis and  construct a method for approximation of
functions in the univariate case.  {In Section
\ref{Multivariate approximation}, we extend the method of
approximation developed in Section \ref{Univariate approximations} to
the multivariate case, in particular, prove  upper bounds for the
approximation error and convergence rate, we also prove a lower bound of $M_n(U_{\lambda, 2}(\bT^d))_2$.}

\section{Univariate  approximation} \label{Univariate approximations}
{
In this section, we  construct a linear method  in the form of a  linear combination of translates of a function $\varphi_{\beta}$   defined  as in (\ref{varphi}) for approximation of univariate functions in $\mathcal{H}_{\lambda, p}(\bT)$.  We give upper bounds of the approximation error for various $\lambda$ and $\beta$.
}

Let   $\lambda, \beta, \vartheta:\ \RR \to \RR$  be given 2-times continuously differentiable  functions and   $\vartheta$ be such that
\begin{equation}\nonumber
\vartheta(x):=
\begin{cases}
1, & \text{ if }  x \in [ -\frac{1}{2}, \frac{1}{2}],\\
0, & \text{ if }  x \not\in  (-1,1).
\end{cases}
\end{equation}
 Corresponding  to these  functions we define the functions { $\mathcal{G}$ and $H_m$} as 
\begin{equation}\label{Hm} 
\mathcal{G}(x):= \frac{\lambda(x)}{\beta(x)}, \quad H_m(x) :=\sum_{k \in \ZZ} \vartheta(k/m)\mathcal{G}(k) \,e^{ik x}.
\end{equation} 
For a function $f \in \mathcal{H}_{\lambda, p} (\TT)$ represented as $f= \varphi_\lambda * g$, $g \in L^p(\TT)$, we define the operator 
\begin{equation}\label{Qm} 
Q_{m,\beta}(f):= \ \frac{1}{2m+1}\sum_{k=0}^{2m} V_m(g) \left(\frac{k}{2m+1}\right)\varphi_\beta \left(\cdot - \frac{k}{2m+1}\right),
\end{equation}
where   
$
V_m(g):=H_m *g.
$
Finally, we define for a function $h: \RR \to \RR$,
 $$\sigma_m (h;f)(x):=\sum_{k\in \mathbb{Z}} h(k/m)\widehat{f}_k e^{ikx}.$$

{Let us obtain upper  estimates} for the error of approximating a function  $f \in \mathcal{H}_{\lambda, p} (\TT)$ by the trigonometric polynomial $Q_{m,\beta}(f)$ a 
linear combination of $2m+1$ translates of the function $\varphi_\beta$. 

\medskip
{
\begin{defn}
	A 2-times continuously differentiable function $\psi: \mathbb{R} \to \mathbb{R}$ is called a function of monotone type if there exists a positive constant $c_0$ such that 
	$$
	|\psi(x)| \geq c_0 |\psi(y)|, \quad |\psi^{''}(x)| \geq c_0 |\psi^{''}(y)| \quad  \text{for all } \ 2|y| \geq |x| \geq |y|/2.
	$$ 
\end{defn}
}
We put 
\begin{equation*}
 \varepsilon_m := J_m ( {\lambda}) +  \sup_{|x|\in [-m,m]}\left(  |\mathcal{G}(x)|  + m^2\sup_{|x|\in [-m,m]} |\mathcal{G}^{''}(x)|\right)  J_m ( {\beta}),
\end{equation*}
where for a 2-times continuously differentiable function $\psi$, 
$$
J_m(\psi):=\int_{|x| \geq m}\left(\left|\frac{\psi(x)}{m}\right| + \Big|x \psi^{''}(x)\Big|\right) dx. 
$$
\begin{theorem} \label{p=1}
Let $1 \leq p \leq  \infty$. Assume that the  functions $\lambda,\beta$ are of monotone type.  Then there exists a positive constant $c$ such that for all  $f\in \mathcal{H}_{\lambda, p} (\TT)$ and $m\in \mathbb{N}$,
\begin{equation*} 
\|f-Q_{m,\beta}(f)\|_p  \leq c \varepsilon_m \|f\|_{\mathcal{H}_{\lambda, p} (\TT)}.
\end{equation*} 
\end{theorem}

Before we give the proof of the above theorem, we recall  a lemma proved in \cite{Mha01}, \cite{Mha02}. 
\begin{lemma}\label{lemma1}
Let $1\leq p\leq \infty$, $f \in L^p(\TT)$ and  $h:\mathbb{R}\to \mathbb{R}$ be  2-times continuously differentiable function, supported on $[-1,1]$. Then there exists a constant $c_1$ independent of $f, h, m$ such that 
 $$\| \sigma_m(h; f)||_p \leq c_1 \|h^{''}\|_\infty \|f\|_p.$$
\end{lemma}

{We also need a Landau's inequality for derivatives \cite{La13}.}
\begin{lemma} \label{lemma2}
Let $f \in L^\infty(\RR)$  be  $2$-times continuously differentiable function. Then 
$$\|f^{'}\|_\infty^2 \leq 4\|f\|_\infty \|f^{''}\|_\infty.$$
In particular,
$$\|f^{'}\|_\infty \leq \|f\|_\infty + \|f^{''}\|_\infty.$$
\end{lemma}

\begin{proof}(Proof of Theorem \ref{p=1})
	Let $f \in \mathcal{H}_{\lambda, p}(\TT)$ be represented as $\varphi_{\lambda,d}*g$ for some $g \in  L^p(\TT)$.
We define { the kernel $P_m(x,t)$}  for $x,t \in \mathbb{T}$ as 
\begin{equation*} 
P_m (x,t)
:= \ 
\frac{1}{2m+1} \sum\limits_{k=0}^{2m}\varphi_\beta \left(x- \frac{k}{2m+1}\right) H_m \left(\frac{k}{2m+1} -t\right). 
\end{equation*} 
It is  easy to   obtain from   the definition (\ref{Qm})  that
\begin{equation} \nonumber
Q_{m,\beta}(f) (x)=\frac{1}{2\pi}\int_{\mathbb{T}} P_m (x,t) g(t)\, dt.
\end{equation} 
We now use  equation (\ref{varphi}),  the definition of the trigonometric polynomial $H_m$ given in equation  (\ref{Hm}) and the easily verified fact,   for $k,s \in \ZZ, s \in [-m,m]$, that
\begin{equation*}
\frac{1}{2m+1}\sum\limits_{\ell =0}^{2m}  e^{i k(t- (\ell/2m+1))}  e^{is((\ell/2m+1) -t)}=
\begin{cases}
0,  \quad& \text { if } \ \frac{k-s}{2m+1} \not\in \mathbb{Z},\\[2ex]
e^{i(k-k_m)t },   \quad & \text{ if } \ \frac{k-s}{2m+1} \in \mathbb{Z},
\end{cases}
\end{equation*}
 to conclude that 
\begin{equation} \nonumber
P_m (x,t)= \ \sum_{k\in \mathbb{Z}}  \gamma(k) {e^{i  k x} } e^{-i k_m t},  
\end{equation} 
where $\gamma(k)=\vartheta (k_m/m)\mathcal{G} ({k_m}) \beta(k)$ and $k_m \in [-m,m]$ satisfy $(k-k_m)/(2m+1) \in \mathbb{Z}$.
Hence,
\[
\begin{split}
Q_{m,\beta}(f)(x)  &=\sum_{k>m}  \gamma(k)  e^{i kx} \widehat{g}({k_m})  + \sum_{k<-m}  \gamma(k)  e^{i kx}  \widehat{g}({k_m})  +  \sum_{k=-m}^m  \gamma(k) e^{i kx}  \widehat{g}({k_m})   \\[2ex] 
&=: \mathcal{A}_m(x) + \mathcal{B}_m(x) +\mathcal{C}_m(x). 
\end{split}
\]
Consequently,
\begin{equation} \label{aa|Q_m(f) -f|_p<}
\|f-Q_{m,\beta}(f) \|_p 
\ \le \  \|\mathcal{A}_m\|_p + \|\mathcal{B}_m\|_p+ 
\|f- \mathcal{C}_m\|_p.
\end{equation}
For each $j \in \NN$,    we define the functions $\Lambda_{j,m} (x), \mathcal{J}_m (x)$, $\mathcal{K}_{j,m}(x)$, $\mathcal{D}_{j,m} (x)$ and {the} set $I_{j,m}$ as follows
{
$$
\Lambda_{j,m} (x):={\beta}(mx + j(2m+1)), \qquad  \mathcal{J}_m (x):= \mathcal{G}(mx), 
$$
$$
\qquad  \mathcal{K}_{j,m}(x):=\Lambda_{j,m} (x) \vartheta(x) \mathcal{J}_m (x), \qquad \mathcal{D}_{j,m} (x):= \sum_{k\in I_{j,m}} \gamma(k) e^{i kx} \widehat{g}(k_m),
$$
$$I_{j,m}: =\{k\in \mathbb{Z}:\quad (2m+1)j - m \leq k\leq (2m+1)j +m  \}.
 $$
}
Then we have
\begin{equation}\label{Am}
\mathcal{A}_m(x)=  
\sum_{j \in \mathbb{N} }\sum_{k\in I_{j,m}} \gamma(k) e^{i kx} \widehat{g}(k_m)=\sum_{j \in \mathbb{N} } \mathcal{D}_{j,m} (x),
\end{equation} and  for all  $k\in I_{j,m}$,
\begin{equation*}
\begin{split}
\gamma(k)&=\beta({k}) \vartheta (k_m/m) \mathcal{G} ({k_m})= {\beta} (j(2m+1)+ k_m) \vartheta (k_m/m)\mathcal{G} ({k_m})
\\[1ex]
&= \Lambda_{j,m} (k_m/m) \vartheta({k_m}/m) \mathcal{G} (k_m) = \Lambda_{j,m} (k_m/m) \vartheta({k_m}/m) \mathcal{J}_m (k_m/m) 
=  \mathcal{K}_{j,m}(k_m/m). 
\end{split}
\end{equation*}
Hence, 
\begin{equation*}
\begin{split} 
\mathcal{D}_{j,m}(x)&=\sum_{k\in I_{j,m}} \gamma(k) e^{i kx} \widehat{g}(k_m) = \sum_{k_m\in [-m,m]} \mathcal{K}_{j,m}(k_m/m) e^{i (j(2m+1)+ k_m )x} \widehat{g}(k_m)  
\\[1ex]
&= e^{i j(2m+1)x} \sum_{k_m\in [-m,m]} \mathcal{K}_{j,m}(k_m/m) e^{i  k_m x} \widehat{g}(k_m)= e^{i j(2m+1)x} \sigma_m (\mathcal{K}_{j,m}; g).
\end{split} 
\end{equation*}
Therefore,  by Lemma \ref{lemma1}, there exists a constant $c_1$ such that  
\begin{equation*}
\begin{split} 
\|\mathcal{D}_{j,m}\|_p \leq  c_1\|(\mathcal{K}_{j,m})^{''}\|_\infty  \|{g}\|_p.
\end{split} 
\end{equation*}
Then it follows from (\ref{Am}) that 
\begin{equation}\label{Dm++}
\| \mathcal{A}_m\|_p\leq \sum_{j \in \mathbb{N} } \| \mathcal{D}_{j,m} \|_p \leq c_1\sum_{j \in \mathbb{N} } \|(\mathcal{K}_{j,m})^{''}\|_\infty  \|{g}\|_p.
\end{equation}
From the definition of $\mathcal{K}_{j,m}$, $\text{supp}\, \vartheta \subset [-1,1]$, and $\|\vartheta\|_\infty \leq 2 \|\vartheta^{'}\|_\infty \leq 4\|\vartheta^{''}\|_\infty,$ we deduce that 
\begin{equation*}
\begin{split} 
 \|(\mathcal{K}_{j,m})^{''}\|_\infty  &\leq 4\|\vartheta^{''}\|_\infty   \sup_{x\in [-1,1]} \left( |\Lambda_{j,m}(x) \mathcal{J}_m (x)| + |(\Lambda_{j,m} \mathcal{J}_m)^{'}(x)| + | (\Lambda_{j,m} \mathcal{J}_m)^{''}(x)| \right) \\[1ex]
&\leq 4\|\vartheta^{''}\|_\infty  \Bigg[\sup_{x\in I_{j,m}} \left(  |{{\beta}}(x)| + m  |{{\beta}^{'}}(x)|  + m^2  |{{\beta}^{''}}(x)| \right) \sup_{x\in [-m,m]} |\mathcal{G}(x)| \\[1ex]
&+ m \sup_{x\in I_{j,m}} \left( |{{\beta}}(x)|+ m |{{\beta}^{'}}(x)| \right) \sup_{x\in [-m,m]} |\mathcal{G}^{'}(x)| +  m^2\sup_{x\in I_{j,m}} |{\beta}(x)|  \sup_{x\in [-m,m]} |\mathcal{G}^{''}(x)|\Bigg]. 
\end{split} 
\end{equation*}
Hence, 
\begin{equation*}
\begin{split} 
 \|(\mathcal{K}_{j,m})^{''}\|_\infty  \leq  4\|\vartheta^{''}\|_\infty  \sup_{x\in I_{j,m}} \left(   |{{\beta}}(x)| \ + \ m  |{{\beta}^{'}}(x)|  + m^2|{{\beta}^{''}}(x)|  \right) 
 \sup_{x\in [-m,m]} \left(  |\mathcal{G}(x)| +\ m  |\mathcal{G}^{'}(x)| +  m^2   |\mathcal{G}^{''}(x)| \right)  
\end{split} 
\end{equation*}
for all $j\in \mathbb{N}$. 
Therefore,  it follows from (\ref{Dm++}) that 
\begin{multline*} 
\| \mathcal{A}_m\|_p  \leq 4c_1\|\vartheta^{''}\|_\infty \sum_{j\in \mathbb{N}}   \sup_{x\in I_{j,m}} \left(  |{{\beta}}(x)| + m |{{\beta}^{'}}(x)|  + m^2  |{{\beta}^{''}}(x)|  \right) \times\\
 \times \sup_{x\in [-m,m]} \left(  |\mathcal{G}(x)| +m  |\mathcal{G}^{'}(x)| +  m^2   |\mathcal{G}^{''}(x)| \right) \|{g}\|_p.
\end{multline*}
So, by Lemma \ref{lemma2}, we have  
\begin{multline}\label{Am1} 
\| \mathcal{A}_m\|_p  \leq 16 c_1\|\vartheta^{''}\|_\infty \sum_{j\in \mathbb{N}}   \sup_{x\in I_{j,m}} \left(  |{{\beta}}(x)|   + m^2  |{{\beta}^{''}}(x)|  \right)  \sup_{x\in [-m,m]} \left(  |\mathcal{G}(x)| +  m^2   |\mathcal{G}^{''}(x)| \right) \|{g}\|_p.
\end{multline}
Since the function $\alpha, \beta$ is of monotone type,  there exists a constant $c_0$
 such that 
\begin{equation}\label{monot}
	|\alpha(x)| \geq c_0 |\alpha(y)|,  |\alpha^{''}(x)| \geq c_0 |\alpha^{''}(y)|, |\beta(x)| \geq c_0 |\beta(y)|,  |\beta^{''}(x)| \geq c_0 |\beta^{''}(y)| 
\end{equation}
for all $\ 4|y| \geq |x| \geq |y|/4.$
Hence,
$$\sup_{|x|\in I_{j,m}} |{\beta}(x) |
\leq \frac{c_0}{m} \int_{|x|\in I_{j,m}} |{\beta}(x)| dx,$$
$$\sup_{|x|\in I_{j,m}} |m^2 {\beta}^{''}(x) |
\leq c_0 m \int_{|x|\in I_{j,m}}  | {\beta}^{''}(x)| dx.$$
So,
$$ \sum_{j\in \NN}  \sup_{|x|\in I_{j,m}}\left( |{\beta}(x) | +  
|m^2 {\beta}^{''}(x) |  \right) 
\leq 
c_0 \int_{|x| \geq m}  \left(\frac{{|\beta}(x)|}{m} +   \big|m{\beta}^{''}(x)\big| \right) dx \leq 
c_0 J_m ({\beta}).
$$
Combining this with  (\ref{Am1}), we obtain that 
\begin{equation}\label{aaDm+}
\| \mathcal{A}_m\|_p\leq  16 c_0 c_1 \|\vartheta^{''}\|_\infty \varepsilon_m   \|{g}\|_p.
\end{equation}
Similarly, 
\begin{equation}\label{aaDm-}
\| \mathcal{B}_m\|_p\leq 16 c_0 c_1 \|\vartheta^{''}\|_\infty \varepsilon_m   \|{g}\|_p.
\end{equation}

\noindent Next, we will estimate $\| f- \mathcal{C}_m\|_p$.
Notice that $\gamma(k)=\vartheta(k/m) \mathcal{G} (k) \beta(k)= \vartheta(k/m) \lambda(k)$ for $k \in [-m,m]$,  and then 
 $$
 \sigma_m(\vartheta; f)(x)=\sum_{k\in \mathbb{Z}} \vartheta(k/m)  \widehat{f}(k)e^{ikx}= \sum_{k=-m}^m \vartheta(k/m) \lambda(k)  \widehat{g}(k)e^{ikx}=\sum_{k=-m}^m \gamma(k)  \widehat{g}(k)e^{ikx}=\mathcal{C}_m (x),
  $$
and therefore,
\begin{equation}\label{f-Cm}
\|f- \mathcal{C}_m \|_p= \|f- \sigma_m(\vartheta; f)\|_p.
\end{equation}
We define the functions  $S(x), \Phi_{j,m}(x)$ and $\Psi_{j,m} (x)$ {as}
$$
S(x):=\vartheta(x)- \vartheta(x/2), \quad 
\Phi_{j,m}(x):={\lambda} (2^jmx), \quad 
\Psi_{j,m} (x):= S(x) \Phi_{j,m} (x).
$$
Clearly, we have that
$$
(\vartheta(k/(2^{j+1}m)) -\vartheta(k/(2^{j}m)))\lambda(k)=S(k/ (2^j m))\Phi_{j,m} (k/(2^jm))=\Psi_{j,m}(k/(2^jm)), 
$$
which together with 
\begin{equation*}
\begin{split} 
&\sigma_{2^{j+1} m}(\vartheta; f)-\sigma_{2^j m}(\vartheta; f)= \sum_{k \in \mathbb{Z}} (\vartheta(k/(2^{j+1}m)) -\vartheta(k/(2^{j}m))\widehat{f}(k)e^{ikx}\\
&=
 \sum_{k \in \mathbb{Z}} (\vartheta(k/(2^{j+1}m)) -\vartheta(k/(2^{j}m)))\lambda(k) \widehat{g}(k)e^{ikx}
\end{split} 
\end{equation*}
implies that
$$\sigma_{2^{j+1} m}(\vartheta; f)-\sigma_{2^j m}(\vartheta; f)= \sum_{k \in \mathbb{Z}} \Psi_{j,m}(k/(2^jm)) \widehat{g}(k)e^{ikx}= \sigma_{2^jm}(\Psi_{j,m}; g).$$
Then by  Lemma \ref{lemma1}, we obtain  
\begin{equation}\label{sigma4}
\|\sigma_{2^{j+1} m}(\vartheta; f)-\sigma_{2^j m}(\vartheta; f)\|_p \leq c_1 \|\Psi_{j,m}^{''}\|_\infty  \|{g}\|_p.
\end{equation}
Moreover, from the definition of $\Psi_{j,m}$, $\text{supp}S\subset [-2,-1/2] \cup [1/2,2]$,  and $\|S\|_\infty \leq 2 \|S^{'}\|_\infty \leq 4\|S^{''}\|_\infty\leq 8\|\vartheta^{''}\|_\infty,$
we have that
\begin{equation*}
\begin{split} 
|\Psi_{j,m}^{''}(x)| &=|S^{''}(x) \Phi_{j,m}(x) + 2S^{'}(x) \Phi_{j,m}^{'}(x) + S(x) \Phi_{j,m}^{''}(x) |
\\[1ex]
& \leq 8\|\vartheta^{''}\|_\infty   \sup_{|x|\in [1/2,2]}\left(  |\Phi_{j,m}(x) |+ \Phi_{j,m}^{'}(x)| +  | \Phi_{j,m}^{''}(x)| \right) \\[1ex]
&\leq 16\|\vartheta^{''}\|_\infty   \sup_{|x|\in [1/2,2]} \left( |\Phi_{j,m}(x) |  +  | \Phi_{j,m}^{''}(x)|\right) \\[1ex]
&=16\|\vartheta^{''}\|_\infty \sup_{|x|\in [2^{j-1} m,2^{j+1}m]} \left(  |{\lambda}(x) |
 +(2^j m)^2  |{\lambda}^{''}(x) | \right)\\[1ex]
&\leq  64\|\vartheta^{''}\|_\infty \sup_{|x|\in [2^{j-1} m,2^{j+1}m]}\left( |{\lambda}(x) |
 +  |x^2{\lambda}^{''}(x) | \right). 
\end{split} 
\end{equation*}
Combining this and (\ref{sigma4}), we deduce 
$$\|\sigma_{2^{j+1} m}(\vartheta; f)-\sigma_{2^j m}(\vartheta; f)\|_p \leq 64 c_1 \|\vartheta^{''}\|_\infty \sup_{|x|\in [2^{j-1} m,2^{j+1}m]}\left( |{\lambda}(x) |
 +  |x^2{\lambda}^{''}(x) | \right)  \|{g}\|_p.$$
Therefore, by (\ref{f-Cm}) and $\lim_{m\to \infty}\| f- \sigma_{2^j m}(\vartheta; f)\|_p=0$, we have that 
\begin{align}\label{f-Cm111}
\notag \| f- \mathcal{C}_m\|_p  &\leq 
\sum_{j=0}^\infty \|\sigma_{2^{j+1} m}(\vartheta; f)-\sigma_{2^j m}(\vartheta; f)\|_p  \\
&\leq 64 c_1 \|\vartheta^{''}\|_\infty\sum_{j=0}^\infty  \sup_{|x|\in [2^{j-1} m,2^{j+1}m]}\left( |{\lambda}(x) |
 +  |x^2{\lambda}^{''}(x) | \right) \|{g}\|_p.
\end{align}
Since (\ref{monot}), 
$$\sup_{|x|\in [2^{j-1} m,2^{j+1}m]} |{\lambda}(x) |
\leq \frac{c_0}{2^jm} \int_{|x|\in [2^j m,2^{j+1}m]} |{\lambda}(x)| dx \leq \frac{c_0}{m} \int_{|x|\in [2^j m,2^{j+1}m]} |{\lambda}(x)| dx,$$
and
$$\sup_{|x|\in [2^{j-1} m,2^{j+1}m]} |x^2 {\lambda}^{''}(x) |
\leq 2 c_0 \int_{|x|\in [2^j m,2^{j+1}m]}  | x{\lambda}^{''}(x)| dx.$$
So,
$$ \sum_{j=0}^\infty  \sup_{|x|\in [2^{j-1} m,2^{j+1}m]}\left( |{\lambda}(x) |+  |x^2 {\lambda}^{''}(x) | 
 \right) \leq 2 c_0 \int_{|x| \geq m}  \left(\frac{{|\lambda}(x)|}{m} +   |x{\lambda}^{''}(x)| \right) dx= 2c_0 J_m ({\lambda}).
$$
Hence, by (\ref{f-Cm111}), we deduce 
 \begin{equation}\label{f-Cm122}
\| f- \mathcal{C}_m\|_p  \leq  128 c_0 c_1 \|\vartheta^{''}\|_\infty \varepsilon_m \|{g}\|_p.
\end{equation}
Combining (\ref{aaDm+}), (\ref{aaDm-}) and (\ref{f-Cm122}) we have 
\begin{equation*} 
\|f-Q_{m,\beta}(f)\|_p \ \leq c\varepsilon_m \|f\|_{\mathcal{H}_{\lambda, p} (\TT)}.
\end{equation*} 
\hfill
\end{proof}

From the above theorem, by letting $\lambda=\beta$, we obtain the following corollary.
\begin{coro} \label{hq1}
Let $1 \leq p \leq  \infty$ and ${\lambda}$  be of monotone type. Then there exists a positive constant $c$ such that for all  $f\in \mathcal{H}_{\lambda, p} (\TT)$ and $m\in \mathbb{N}$,
\begin{equation*}
\|f-Q_{m,\lambda}(f)\|_p  \leq c J_m(\lambda) \|f\|_{\mathcal{H}_{\lambda, p} (\TT)}.
\end{equation*} 
\end{coro}

\begin{defn}
Let $r, \kappa \in \RR$. A function $f: \RR \to \RR$ will be called a mask of  type $(r,\kappa)$ if $f$ is an even,  $2$ times continuously differentiable such that for $t\geq 1$, $f(t)= |t|^{-r} (\log (|t|+1))^{-\kappa} F (\log |t|)$ for some $F : \RR \to \RR$ such that $|F^{(k)} (t)| \leq a_1$ for all  $t\geq 1, k=0,1,2.$
\end{defn}

\begin{theorem} \label{weakype}
Let $1 \leq p \leq  \infty$, $1<r<\infty, \kappa\in \RR $ and the function ${\lambda}$  be a mask of  type $(r,\kappa)$. Then there exists a positive constant $c$ such that for all  $f\in \mathcal{H}_{\lambda, p} (\TT)$ and $m\in \mathbb{N}$,
\begin{equation*}
\|f-Q_{m,\lambda}(f)\|_p  \leq c m^{-r} (\log m)^{-\kappa} \   \|f\|_{\mathcal{H}_{\lambda, p} (\TT)}.
\end{equation*} 
\end{theorem}
\begin{proof}
Since the function ${\lambda}$  be a mask of  type $(r,\kappa)$ and $r>1$, 
\begin{equation}\label{ab1}
 \int_{|x| \geq m}  \left| \frac{{\lambda}(x)}{m}\right| dx   \leq a_1  \int_{|x| \geq m}  \frac{|x|^{-r} (\log (|x|+1))^{-\kappa}}{m} dx \leq a_2 m^{-r} (\log (m+1))^{-\kappa} \quad \forall m \in \NN.
\end{equation}
On the other hand,
\begin{equation*}
\begin{split}
& \int_{|x| \geq m}  | x{\lambda}^{''}(x)| dx   \leq    \int_{|x| \geq m}  |x| \ \Big( (|x|^{-r} (\log (|x|+1))^{-\kappa})^{''} |F (\log |x|) | \\[1ex] 
&+  2(|x|^{-r} (\log (|x|+1))^{-\kappa})^{'} |F^{'} (\log |x|) | / |x| 
+ (|x|^{-r} (\log (|x|+1))^{-\kappa})  |F^{''} (\log |x|) -F^{'} (\log |x|) |/ x^2   \Big) dx\\[1ex]
&\leq a_1\int_{|x| \geq m}  |x| \left( (|x|^{-r} (\log (|x|+1))^{-\kappa})^{''} 
+  2(|x|^{-r} (\log (|x|+1))^{-\kappa})^{'}/|x|   
+ 2(|x|^{-r} (\log (|x|+1))^{-\kappa})/ x^2    \right) dx \\[1ex]
&\leq a_3 m^{-r} (\log (m+1))^{-\kappa}.
\end{split}
\end{equation*}
Hence, by (\ref{ab1}), we deduce 
$$J_m({\lambda} )\leq a_4 m^{-r}(\log (m+1))^{-\kappa}. $$
From this and Corollary \ref{hq1},  we complete the proof.
\end{proof}

 \begin{coro} 
For $1 \leq p \leq  \infty$, $1<r<\infty$ and $\lambda (x)=\beta (x)= x^{-r}$ for $x \ne 0$,  $\mathcal{H}_{\lambda, p} (\TT)$ becomes the Korobov space $K^r_p(\bT)$.
Then  we have the estimate as in \cite{1}:
$$
M_n(U_{\lambda, p}(\bT), \kappa_{r})_p \leq c m^{-r}$$
where $\kappa_{r}$ is the Korobov function.
\end{coro}

\begin{defn}
A function $f: \RR \to \RR$ is called a function of exponent  type if $f$ is  $2$ times continuously differentiable and there exists a positive  constant  $s$ such that $f(t)= e^{-s |t| } F (|t|)$ for some decreasing function $F: [0, +\infty ) \to (0, +\infty ).$ 
\end{defn}

\begin{theorem} 
Let $1 \leq p \leq  \infty,$ $1<r<\infty,$ $\kappa\in \ZZ$, the function   ${\lambda}$  be a mash of type $(r,\kappa)$, the function   ${\beta}$ of exponent  type. Then there exists a positive constant $c$ such that for all  $f\in \mathcal{H}_{\lambda, p} (\TT)$ and $m\in \mathbb{N}$, we have 
\begin{equation} \nonumber
\|f-Q_{m,\beta}(f)\|_p  \leq c m^{-r} (\log (m+1))^{-\kappa} \|f\|_{\mathcal{H}_{\lambda, p} (\TT)}.
\end{equation} 
\end{theorem}

\begin{proof}
We will use the notation in the proof of Theorem \ref{p=1}.	
For $k\in I_{j,m}$ we have $k_m= k - j(2m+1)$ and then 
\begin{equation*}
\begin{split}
| \gamma(k) |&= \Bigg|\beta({k_m +j(2m+1)}) \vartheta({k_m}/m) \frac{\lambda({k_m})}{\beta({k_m})}\Bigg| \\[1.5ex]
&=e^{-sj (2m+1)) } \frac{|\lambda({k_m}) F({k_m +j(2m+1)}) |}{|F({k_m})|} \leq  b_1 e^{-sj (2m+1)) }. 
\end{split}
\end{equation*}
Hence,
$$\Bigg\| \sum_{k\in I_{j,m}} \gamma(k) e^{i kx} \widehat{g}(k_m)\Bigg\|_p \ \leq \ 3 b_1 m e^{-sj (2m+1)) } \|{g} \|_p.$$
This implies that
\begin{equation}\label{aaDm+22}
\begin{split}
\| \mathcal{A}_m\|_p&=\Bigg\| \sum_{j \in \mathbb{N} }\sum_{k\in I_{j,m}} \gamma(k) e^{i kx} \widehat{g}(k_m)\Bigg\|_p 
\\[1.5ex]
&\leq 3 b_1\sum_{j \in \mathbb{N} }  me^{-sj (2m+1)) } \|{g} \|_p  \leq b_2 m^{-r}(\log (m+1))^{-\kappa} \|{g}\|_p.
\end{split}
\end{equation}
Similarly, 
\begin{equation}\label{aaDm-22}
\| \mathcal{B}_m\|_p\leq b_2 m^{-r} (\log (m+1))^{-\kappa} \|{g}\|_p.
\end{equation}
We also known that in the proof of Theorem \ref{p=1} that  
\begin{equation}\label{f-Cm1}
\| f- \mathcal{C}_m\|_p  \leq  b_3\sum_{j=0}^\infty  \sup_{|x|\in [2^{j-1} m,2^{j+1}m]}\left( |{\lambda}(x) |
 +  |x^2{\lambda}^{''}(x) | \right) \|{g}\|_p.
\end{equation}
We see that
$$\sup_{|x|\in [2^{j-1} m,2^{j+1}m]} |{\lambda}(x) |
\leq b_4 \int_{|x|\in [2^j m,2^{j+1}m]} \frac{|{\lambda}(x)|}{|x|} dx $$
$$\sup_{|x|\in [2^{j-1} m,2^{j+1}m]} |x^2 {\lambda}^{''}(x) |
\leq b_4 \int_{|x|\in [2^j m,2^{j+1}m]}  | x{\lambda}^{''}(x)| dx.$$
So,
\begin{equation*}
\begin{split}
 \sum_{j=0}^\infty  \sup_{|x|\in [2^{j-1} m,2^{j+1}m]}\left( |{\lambda}(x) |+  |x^2 {\lambda}^{''}(x) | 
 \right) \leq  b_4 \int_{|x| \geq m}  \left(\frac{{|\lambda}(x)|}{|x|} +   \big|x{\lambda}^{''}(x)\big| \right) dx.
\end{split}
\end{equation*}
Hence, by (\ref{f-Cm1}), we deduce that
\begin{equation*} 
\| f- \mathcal{C}_m\|_p  
  \leq   
   b_3 b_4\|{g}\|_p \int_{|x| \geq m}  \left(\frac{{|\lambda}(x)|}{|x|} +   \big|x{\lambda}^{''}(x)\big| \right) dx
\leq 
b_5 m^{-r} (\log (m+1))^{-\kappa} \|{g}\|_p.
\end{equation*}
Combining this, (\ref{aaDm+22}), (\ref{aaDm-22}) and \eqref{aa|Q_m(f) -f|_p<}, we complete    the proof.
\hfill
\end{proof}

\section{Multivariate approximation} \label{Multivariate approximation}
{In this section, we make use of the univariate operators  $Q_{m,\lambda}$ to construct multivariate operators on sparse Smolyak grids  for approximation of functions from  $\mathcal{H}_{\lambda,p}(\TT^d)$.  Based on this approxiation with certain restriction on the function $\lambda$ we prove an upper bound of $M_n(U_{\lambda, p}(\bT^d), \varphi_{\lambda,d})_p$ for $1\le p \le \infty$ as well as a lower bound of 
$M_n(U_{\lambda, 2}(\bT^d))_2$. The results obtained in this section generalize some results in \cite{1,DM-Err16}.}
	
\subsection{Error estimates for functions in the space $\mathcal{H}_{\lambda,p}(\TT^d)$}

For $\bm \in \NNd$, let the multivariate operator $Q_\bm$ in $\mathcal{H}_{\lambda,p}(\TT^d)$ be defined by
\begin{equation} \label{def[Q_m(d>1)]}
Q_\bm:= \prod_{j=1}^d  Q_{m_j,\lambda},
\end{equation}
where the univariate operator
$Q_{m_j,\lambda}$ is applied to the univariate function $f$ by considering $f$ as a
function of  variable $x_j$ with the other variables held fixed,
$\ZZdp:= \{\bk\in \ZZd: \ k_j \ge 0, \ j \in \mathbb{N}_d\}$ and $k_j$ denotes the $j$th coordinate of $\bk$.

Set $\ZZd_{-1}:= \{\bk\in \ZZd: \ k_j \ge -1, \ j \in \mathbb{N}_d\}$.  For $k \in \ZZ_{-1}$, we define the univariate operator $T_k$ in $\mathcal{H}_{\lambda,p}(\TT)$ by
\begin{equation*}
T_k := {{\rm I}}-  Q_{2^k,\lambda}, \ k \ge 0, \quad  T_{-1}:= {\rm I},
\end{equation*}
where ${{\rm I}}$ is the identity operator. If $\bk\in \ZZd_{-1}$, we define the mixed operator
$T_\bk$ in $\mathcal{H}_{\lambda,p}(\TT^d)$ in the manner of the definition of \eqref{def[Q_m(d>1)]} as
\begin{equation*}
T_\bk:= \prod_{i=1}^d  T_{k_i}.
\end{equation*}
Set $|\bk|:= \sum_{j\in \mathbb{N}_d} |k_j|$ for  $\bk\in \ZZd_{-1}$ and $\bk_{(2)}^{-\kappa}=\prod_{j=1}^d (k_j+2)^{-\kappa}$.

\begin{lemma}  \label{lemma[T_k(f)<]}
Let $1 \leq p \leq \infty$, $1<r<\infty, 0\leq \kappa<\infty $  and the function ${\lambda}
$  be a mask of  type $(r,\kappa)$.
Then we have for any $f \in \mathcal{H}_{\lambda,p}(\TT^d)$ and $\bk\in \ZZd_{-1}$,
\begin{equation} \nonumber
\| T_\bk(f)\|_p
\ \le \
C \bk_{(2)}^{-\kappa} 2^{-r|\bk|}  \|f\|_{\mathcal{H}_{\lambda,p}(\TT^d)}
\end{equation}
with some constant $C$ independent of $f$ and $\bk$.
\end{lemma}

\begin{proof}
We prove the lemma by induction on $d$. For $d=1$ it follows from Theorems \ref{weakype}. Assume the lemma is true for $d-1$. Set ${\bx}':= \{x_j: \, j \in \NN_{d-1}\}$ and ${\bx} = ({\bx}',x_d)$ for ${\bx} \in \RRd$.
We temporarily denote by $\|f\|_{p,{\bx}'}$ and $\|f\|_{\mathcal{H}_{\lambda,p}(\TT^{d-1}), {\bx}'}$ or $\|f\|_{p,x_d}$ and $\|f\|_{\mathcal{H}_{\lambda,p}(\TT),x_d}$ the norms  applied to the function $f$ by considering $f$ as a function of  variable
${\bx}'$ or $x_d$ with the other variable held fixed, respectively.
For $\bk=(\bk',k_d) \in \ZZd_{-1}$, we get by Theorems \ref{weakype}  and the induction assumption
\begin{equation} \nonumber
\begin{aligned}
\|T_\bk(f)\|_p
\ &  = \
\|\|T_{\bk'}T_{k_d}(f)\|_{p,{\bx}'}\|_{p,x_d}
\   \ll \
\| 2^{-r|\bk'|} {\bk'}_{(2)}^{-\kappa} \|T_{k_d}(f)\|_{\mathcal{H}_{\lambda,p}(\TT^{d-1}), {\bx}'}\|_{p,x_d} \\
\ &  = \
2^{-r|\bk'|} {\bk'}_{(2)}^{-\kappa} \| \|T_{k_d}(f)\|_{p,x_d}\|_{\mathcal{H}_{\lambda,p}(\TT^{d-1}), {\bx}'}\\
\ &   \ll \
2^{-r|\bk'|} {\bk'}_{(2)}^{-\kappa} \| 2^{-rk_d} (k_d+2)^{-\kappa}\|f\|_{\mathcal{H}_{\lambda,p}(\TT),x_d}\|_{\mathcal{H}_{\lambda,p}(\TT^{d-1}), {\bx}'} \\
\ &  = \
2^{-r|\bk|} \prod_{j=1}^d (k_j+2)^{-\kappa}\| f\|_{\mathcal{H}_{\lambda,p}(\TT^d)}.
 \end{aligned}
\end{equation}
\hfill
\end{proof}

Let the univariate operator $q_k$ be defined for $k \in \ZZ_+$, by
\begin{equation*}
q_k := \ Q_{2^k,\lambda} -  Q_{2^{k-1},\lambda}, \ k >0, \ \ q_0:= \ Q_{1,\lambda},
\end{equation*}
and in the manner of the definition of \eqref{def[Q_m(d>1)]}, the multivariate operator
$q_\bk$ for $\bk\in \ZZdp$, by

\begin{equation} \nonumber
q_\bk:= \prod_{j=1}^d  q_{k_j}.
\end{equation}
For $\bk\in \ZZdp$, we write $\bk\to \infty$ if $k_j \to \infty$ for each $j \in \mathbb{N}_d$.

\begin{theorem}  \label{theorem[decomposition]}
Let $1 \leq p \leq  \infty$, $1<r<\infty, 0\leq \kappa<\infty $  and the function ${\lambda}
$  be a mask of  type $(r,\kappa)$.
Then every $f \in \mathcal{H}_{\lambda,p}(\TT^d)$ can be represented as the series
\begin{equation} \label{eq[decomposition]}
f
\ = \
\sum_{\bk\in \ZZdp} q_\bk(f)
\end{equation}
converging in $L^p$-norm,
and we have for $\bk\in \ZZdp$,
\begin{equation} \label{ineq[q_k]}
\| q_\bk(f)\|_p
\ \le \
C  2^{-r|\bk|} \bk_{(2)}^{-\kappa} \|f\|_{\mathcal{H}_{\lambda,p}(\TT^d)}
\end{equation}
with some constant $C$ independent of $f$ and $\bk$.
\end{theorem}

\begin{proof}
Let $f \in \kpd$. In a way similar to the proof of Lemma
\ref{lemma[T_k(f)<]}, we can show that
\begin{equation} \nonumber
\|f - Q_{2^\bk}(f)\|_p
\ \ll \
 \max_{j \in \mathbb{N}_d} 2^{-r k_j}k_j^\kappa \|f\|_{\mathcal{H}_{\lambda,p}(\TT^d)},
\end{equation}
and therefore,
\[
\|f -  Q_{2^\bk}(f)\|_p  \to 0 , \ \bk \to  \infty,
\]
 where $2^\bk= (2^{k_j}:\ j \in \mathbb{N}_d)$.
On the other hand,
\begin{equation*}
Q_{2^\bk}
\ = \
\sum_{s_j \le k_j, \ j \in \mathbb{N}_d} \, q_{\bs}(f).
\end{equation*}
This proves \eqref{eq[decomposition]}.  To prove \eqref{ineq[q_k]} we notice that
from the definition it follows that
 \begin{equation*}
 q_\bk
\  = \
\sum_{e \subset \mathbb{N}_d} (-1)^{|e|} T_{\bk^e},
\end{equation*}
where $\bk^e$ is defined by $k^e_j = k_j$ if $j \in e$, and
$k^e_j = k_j - 1$ if $j \notin e$.
 Hence, by Lemma  \ref{lemma[T_k(f)<]}
 \begin{equation*}
\|q_\bk(f)\|_p
\ \le \
\sum_{e \subset \mathbb{N}_d} \| T_{\bk^e}(f)\|_p
\ \ll \
\sum_{e \subset \mathbb{N}_d} 2^{-r|\bk^e|} ({\bk}^e_{(2)})^{-\kappa}\| f\|_{\mathcal{H}_{\lambda,p}(\TT^d)}
\ \ll \
2^{-r|\bk|} \bk_{(2)}^{-\kappa}\| f\|_{\mathcal{H}_{\lambda,p}(\TT^d)}.
\end{equation*}
\hfill
\end{proof}

For approximation of $f \in \mathcal{H}_{\lambda,p}(\TT^d)$, we introduce the linear operator $P_m, m \in \NN$,  by
\begin{equation} \label{P_m}
P_m(f)
:= \
\sum_{|\bk| \le m} q_\bk(f).
\end{equation}
We give an upper bound for the error of the approximation of functions $f \in \mathcal{H}_{\lambda,p}(\TT^d)$ by the operator $P_m$ in the following theorem.

\begin{theorem}  \label{theorem[approximation]}
Let $1 \leq p \leq \infty$, $1<r<\infty, 0\leq \kappa<\infty $  and the function ${\lambda}
$  be a mask of  type $(r,\kappa)$.
Then, we have for every $m \in \NN$ and $f \in \mathcal{H}_{\lambda,p}(\TT^d)$,
\begin{equation} \nonumber
\|f - P_m(f)\|_p
\ \le \
C \, 2^{-rm} m^{d-1-\kappa }\, \|f\|_{\mathcal{H}_{\lambda,p}(\TT^d)}
\end{equation}
with some constant $C$ independent of $f$ and $m$.
\end{theorem}

\begin{proof}
From Theorem \ref{theorem[decomposition]} we deduce that
\begin{equation*}
\begin{aligned}
\|f - P_m(f)\|_p
\ &= \
\Bigg\|\sum_{|\bk| > m} q_\bk(f)\Bigg\|_p
\ \le \
\sum_{|\bk| > m} \|q_\bk(f)\|_p \\[1ex]
\ &\ll \
\sum_{|\bk| > m} 2^{-r|\bk|} \bk_{(2)}^{-\kappa}\|f\|_{\mathcal{H}_{\lambda,p}(\TT^d)}
\ \ll \
\|f\|_{\mathcal{H}_{\lambda,p}(\TT^d)}
\sum_{|\bk| > m} 2^{-r|\bk|}\bk_{(2)}^{-\kappa}\\[1ex]
\ &\ll \
2^{-rm} m^{d-1 - \kappa}\, \|f\|_{\mathcal{H}_{\lambda,p}(\TT^d)}.
\end{aligned}
\end{equation*}
\hfill 
 \end{proof}

\subsection{Convergence rate} \label{Convergence rate and optimality}

We choose a positive integer $m\in\bN$, a lattice vector $\bk \in\bZ^d_+$ with $|\bk|\leq m$ and another lattice vector
$\bs=(s_j:j\in \mathbb{N}_d)\in \prod_{j\in \mathbb{N}_d} Z[2^{k_j+1}+1]$
to define the vector $\by_{\bk,\bs}=\left(\frac{2\pi s_j}{2^{k_j+1}+1}:j\in \mathbb{N}_d\right)$. The
Smolyak grid on $\bT^d$ consists of all such vectors and is given as
$$
G^d(m) := \ \left\{\by_{\bk,\bs}:|\bk|\leq m,
\bs\in\otimes_{j\in \mathbb{N}_d} Z[2^{k_j+1}+1]\right\}.
$$
A simple computation confirms, for $m\rightarrow\infty$ that
$$
|G^d(m)| = \sum_{|\bk| \le m} \prod_{j\in \mathbb{N}_d} (2^{k_j
+ 1} + 1)\asymp2^dm^{d-1},
$$
so, $G^d(m)$ is a sparse subset of a
full grid of cardinality $2^{dm}$. Moreover, by the definition of
the linear operator $P_m$ given in equation \eqref{P_m} we see that
the range of $P_m$ is contained in the subspace
$$
{\rm span} \{\varphi_{\lambda,d}(\cdot-{\by}):{\by}\in G^d(m)\}.
$$
Other words, $P_m$ defines a multivariate method of approximation by translates
of the  function $\varphi_{\lambda,d}$ on the sparse Smolyak grid $G^d(m)$. An upper bound for the error of this approximation of functions from $\kpd$ is given in Theorem \ref{theorem[approximation]}.

Now, we are ready to prove the next theorem, thereby establishing
an  upper  bound of $M_n(U_{\lambda, p}, \varphi_{\lambda,d})_p$.
\begin{theorem} \label{theorem[M_n(U^r_p)_p]}
If $1 \leq  p \leq  \infty$, $1<r<\infty, 0\leq\kappa<\infty $  and the function ${\lambda}
$  be a mask of  type $(r,\kappa)$, then
\begin{equation} \nonumber
M_n(U_{\lambda, p}(\bT^d), \varphi_{\lambda,d})_p \ \ll \ n^{-r} (\log n)^{r(d-1)-\kappa}.
\end{equation}
\end{theorem}

\begin{proof} If $n\in\bN$ and $m$ is the largest positive integer such that
$ |G^d(m)| \le n$, then $n\asymp2^mm^{d-1}$ and by Theorem
\ref{theorem[approximation]} we have that 
$$
M_n(U_{\lambda, p}(\bT^d),\varphi_{\lambda,d})_p
\leq 
\sup_{f\in U_{\lambda, p}(\bT^d)}\|f-P_m(f)\|_{p}
\ll 
2^{-rm}m^{d-1 - d\kappa}\asymp n^{-r}(\log n)^{r(d-1)- \kappa}.
$$
\hfill
\end{proof}

For $p=2$, we are able to establish a lower bound for 
$M_n(U_{\lambda, 2}(\bT^d), \varphi_{\lambda,d})_2$. We prepare some auxiliary results. Let $\mathbb{P}_q(\bR^l)$ be
the set of algebraic polynomials on $\RR^l$ of total degree at most $q$, and 
$$
\EEm:= \{\bt =
(t_j:j\in \NN_m): |t_j|=1, j \in \NN_m\}.
$$  
 We define the polynomial maifold
\begin{equation*}
\mathbb{M}_{m,l,q} := \ \left\{(p_j(\bu):j\in \NN_m):p_j\in\mathbb{P}_q(\bR^l),j\in \NN_m,\bu \in\bR^l\right\}.
\end{equation*}
Denote by $\|\bx\|_2$ the Euclidean norm of a vector $\bx$ in $\bR^m$.  The following lemma was proven in \cite{Ma05}.

\begin{lemma} \label{Lemma[forlowerbound]}
 Let $m, l, q\in \bN$ satisfy the inequality
$l \log(\frac{4emq}{l})\le \frac{m}{4}$.  Then there is a vector $\bt \in \EEm$ and a positive constant $c$ such that
\begin{equation*}
\inf \left\{ \|\bt- \bx\|_2:\bx \in \mathbb{M}_{m,l,q} \right\}
\ \ge \ 
c\, m^{1/2}.
\end{equation*}
\end{lemma}

\begin{theorem}
If  $1<r<\infty, 0\leq\kappa<\infty $ and the function ${\lambda}
$  be a mask of  type $(r,\kappa)$, then we have that
\begin{equation} \label{[M_n(U^r_2)_2]}
n^{-r} (\log n)^{r(d-2)-d\kappa} \ \ll \ M_n(U_{\lambda, 2})_2 \ \ll \ n^{-r} (\log
n)^{r(d-1)-\kappa}.
\end{equation}
\end{theorem}

\begin{proof}
The upper bound of  \eqref{[M_n(U^r_2)_2]} is  in Theorem
\ref{theorem[M_n(U^r_p)_p]}. Let us   prove the lower
bound by {developing} a technique used in the proofs of \cite[Theorem
1.1]{Ma05} and  \cite[Theorem 4.4]{1} . For a positive number $a$ we define a subset
$\bH(a)$ of lattice vectors by
\begin{equation*}
\bH(a) := \ \Bigg\{\bk=(k_j:j\in \mathbb{N}_d)\in\bZ^d: \,
\prod_{j\in \mathbb{N}_d} |k_j| \le a \Bigg\}.
\end{equation*}
Notice that  $|\bH(a)|\asymp
a(\log a)^{d-1}$ when $a\rightarrow\infty$. To apply Lemma
\ref{Lemma[forlowerbound]},  for any $n\in\bN$, we take  $q=\lfloor n(\log n)^{-d+2}\rfloor+1$,
$m=5(2d+1)\lfloor n\log n\rfloor$ and
$l=(2d+1)n$. With these choices we obtain 
\begin{equation}\label{H(q)}
|\bH(q)|\asymp m
\end{equation}
and
\begin{equation}\label{q}
q\asymp m(\log m)^{-d+1}
\end{equation}
 as $n\rightarrow\infty$.
Moreover, we have that
$$
\lim_{n\rightarrow\infty}\frac{l}{m}\log\left(\frac{4emq}{l}\right)=\frac{1}{5},
$$
and therefore, the assumption of Lemma \ref{Lemma[forlowerbound]} is
satisfied for $n\rightarrow\infty$.

Now, let us specify  the polynomial manifold
$\mathbb{M}_{m,l,q}$. To this end, we  put $\zeta := q^{-r} m^{-1/2}(\log q)^{-d\kappa} $ and let $\mathbb{Y}$ be the set
of trigonometric polynomials on $\TTd$, defined by
\begin{equation*}
\mathbb{Y} := \ \Bigg\{f = \zeta \sum_{\bk \in \bH(q)} a_\bk t_\bk
: {\mathbf t}= (t_\bk:\bk \in
\bH(q))\in\bE^{|\bH(q)|}\Bigg\}.
\end{equation*}
If $f\in\mathbb{Y}$ and 
\[
f=\zeta\sum_{\bk
\in \bH(q)} a_\bk t_\bk ,
\]
 then $f=\varphi_{\lambda,d}*g$ for
some trigonometric polynomial $g$ such that
\[
\|{g}\|_{L^2(\bT^d)}^2\leq\zeta^2\sum_{\bk\in\bH(q)}|\lambda (\bk)|^{-2}.
\]
 Since
\begin{equation*}
\begin{split}
\zeta^2\sum_{\bk\in\bH(q)}|\lambda(\bk)|^{-2}
&\leq 
\zeta^2q^{2r}\sum_{\bk\in\bH(q)}\Bigg|\log {\prod_{j=1}^d k_j}\Bigg|^{2\kappa}\\[1ex]
&\leq 
\zeta^2q^{2r}\sum_{\bk\in\bH(q)}\Bigg|\sum_{j=1}^n\log {k_j}\Bigg|^{2d\kappa}
\leq 
\zeta^2 q^{2r} (\log q)^{2d\kappa}|\bH(q)|=m^{-1}|\bH(q)|,
\end{split}
\end{equation*}
by (\ref{H(q)}) that there is a positive constant
$c$ such that $\|{g}\|_{L^2(\bT^d)}\leq c$ for all $n\in\bN$. Therefore, we
can either adjust functions in $\bY$ by dividing them by $c$, or we
can assume without loss of generality that $c=1$, and obtain $\bY\subseteq U_{\lambda,2}(\bT^d)$.

We are now ready to prove the lower bound for $M_n(U_{\lambda,2}(\bT^d))_2$.
We choose any $\varphi\in L^2(\bT^d)$ and let $v$ be any function
formed as a linear combination of $n$ translates of the function
$\varphi$:
\[
v=\sum_{j\in \NN_n}c_j\varphi(\cdot-{\by}_j).
\]
 By the {well-known} Bessel
inequality we  have for a function
\[
f=\zeta\sum_{\bk\in \bH(q)}a_\bk t_\bk\in\bY,
\]
 that
\begin{equation}
\label{f-v}\|f-v\|^2_{L^2(\bT^d)}\geq \zeta^2\sum_{\bk \in\bH(q)}\Bigg|t_\bk-\frac{\widehat{\varphi}(\bk)}{\zeta}\sum_{j\in \NN_n}c_j e^{i({\by}_j,\bk)}\Bigg|^2.
\end{equation}

We introduce a polynomial manifold so that we can use Lemma
\ref{Lemma[forlowerbound]} to get a lower bound for the expressions
on the left hand side of inequality (\ref{f-v}). To this end, we
define the vector ${\mathbf c}=(c_j:j\in \NN_n)\in\bR^n$ and for each
$j\in \NN_n$, let $\bz_j=(z_{j,l}:l\in \mathbb{N}_d)$ be a vector in
$\bC^d$ and then concatenate these vectors to form the vector
$\bz=(\bz_j:j\in \NN_n)\in\bC^{nd}$. We employ the
standard multivariate notation
\[
\bz_j^\bk=\prod_{l\in \mathbb{N}_d}z_{j,l}^{k_l}
\]
and require vectors $\bw=(\bc, \bz)\in\bR^n\times\bC^{nd}$ and $\bu=(\bc, \Re\bz, \Im\bz)\in\bR^l$ to be written in concatenate form. Now, we
introduce for each $\bk\in\bH(q)$ the polynomial $\bq_\bk$
defined at $\bw$ as
$$
\bq_\bk(\bw):=\frac{\widehat{\varphi}(\bk)}{\zeta}\sum_{\bj\in\bH(q)}c_\bj\bz^\bj.
$$
We only need to consider the real part of $\bq_\bk$, namely,
$\bp_\bk=\Re \bq_\bk$ since we have that
\[
\inf\Bigg\{\sum_{\bk\in\bH(q)}\Bigg|t_\bk-\frac{\widehat{\varphi}(\bk)}{\zeta}
\sum_{j\in \NN_n}c_je^{i({\by}_j,\bk)}\Bigg|^2:c_j\in\bR,\by_j\in\bT^d\Bigg\}
\geq
\inf\Bigg\{\sum_{\bk \in\bH(q)}\left|t_\bk-p_\bk(\bu)\right|^2:\bu \in\bR^l\Bigg\}.
\]
Therefore, by Lemma \ref{Lemma[forlowerbound]} and \eqref{q} we conclude there is
a vector $\bt^0= (t^0_\bk:\bk \in
\bH(q))\in\bE^{h_q}$ and the corresponding function
\[
f^0=\zeta\sum_{\bk \in \bH(q)}t^0_\bk \chi_\bk\in\bY
\]
 for which there is a positive constant $c$ such that for
every $v$ of the form
\[
v=\sum_{j\in \NN_n}c_j\varphi(\cdot-{\by}_j),
\] we have that
\[
\|f^0-v\|_{L^2(\bT^d)}\geq c\zeta m^{\frac{1}{2}}=q^{-r} (\log q)^{-d\kappa}\asymp
n^{-r}(\log n)^{r(d-2)-d\kappa}
\]
which proves the lower bound of  \eqref{[M_n(U^r_2)_2]}.
\hfill
\end{proof}

Similar to the proof of the above theorem, we can prove the following theorem for the case $-\infty<\kappa<0$. 
\begin{theorem}
If  $1<r<\infty, -\infty<\kappa<0$ and the function ${\lambda}
$  be a mask of  type $(r,\kappa)$, then we have that
\begin{equation*} 
n^{-r} (\log n)^{r(d-2)-\kappa} \ \ll \ M_n(U_{\lambda, 2} (\bT^d))_2 \ \ll \ n^{-r} (\log
n)^{r(d-1)-d\kappa}.
\end{equation*}
\end{theorem}

\bigskip
\noindent
{\bf Acknowledgments.}
This work  is funded by Vietnam National Foundation for Science and Technology Development (NAFOSTED) under  Grant No. 102.01-2020.03. A part of this work was done when  Dinh {D\~ung} was working at the Vietnam Institute for Advanced Study in Mathematics (VIASM). He would like to thank  the VIASM  for providing a fruitful research environment and working condition.

\end{document}